\documentclass[ letterpaper]{article}
\usepackage[utf8]{inputenc}
\usepackage{blindtext}
\usepackage{amsmath, cite, amsfonts, amsthm}
\usepackage{graphicx}
\usepackage{mathtools}
\usepackage[normalem]{ulem}
\usepackage{tikz}

\usepackage{dsfont}
\usepackage[]{geometry}

\newcommand{\gridThreeD}[5]
{
    \begin{scope}
        \draw [#3,step=1cm] (#1,#2) grid (#4+#1,#5+#2);
    \end{scope}
}

\newcommand{\overbar}[1]{\mkern 1.5mu\overline{\mkern-1.5mu#1\mkern-1.5mu}\mkern 1.5mu}

\theoremstyle{definition}
\newtheorem{theorem}{Theorem}[section]

\newtheorem{lemma}[theorem]{Lemma}

\newtheorem{proposition}{Proposition}[theorem]

\title{A Proof of the Optimal Leapfrogging Conjecture}
\author{Sam K. Miller, Arthur T. Benjamin}
\date{August 7, 2023}

\begin{document}
\maketitle

\section{Introduction}
Suppose we have some checkers placed in the lower left corner of a Go board, and we wish to move them to the upper right corner in as few moves as possible. There are no opponent pieces present, and the pieces move as they would in the game of Chinese Checkers, where for one move, a piece may either shift one unit in any direction, or repeatedly leapfrog over other pieces. \\

Let us consider the Go board as a subset of the non-negative integer lattice $\mathds{Z}^2$. As an example, suppose we have four pieces placed at the coordinates $(0,0)$, $(1,0)$, $(0,1)$, and $(1,1)$, and wish to move them to the squares $(9,9)$, $(10,9)$, $(9,10)$, and $(10,10)$. For the pieces to complete the task in as few moves as possible, the pieces must first be moved into a configuration such that they may jump over each other in an optimal way.\\

We may intuitively attempt lining the checkers up diagonally in what we will call a \textit{snake configuration}, that is, moving the pieces to coordinates $(0,0)$, $(1,1)$, $(2,2)$, and $(3,3)$. By repeating the three-move process of shifting the backmost piece to the right [$(0,0) \xrightarrow{} (1,0)$], leapfrogging that piece to the front [$(1,0) \xrightarrow{} (3,4)$], then shifting it right again [$(3,4) \xrightarrow{} (4,4)$], we can reach our destination in $4 + 4 + (3\times7) = 29$ moves. 
We say that the snake configuration has a speed of 2/3 since in 3 moves, it makes a forward progress of 2. (Every piece is shifted in the direction (1, 1).) \\

However a faster method exists. We first move the pieces into what we call a \textit{serpent configuration}, with the pieces on coordinates $(0,0)$, $(1,0)$, $(1,1)$, and $(2,1)$. Then we repeat the two-move process of leapfrogging the backmost piece to the front [$(0,0) \xrightarrow{} (2,2)$] then leapfrogging the new backmost piece to the front again [$(1,0) \xrightarrow{} (3,2)$], we may reach our destination in $1 + 1 + (2 \times 8) = 18$ moves. This is indeed the fastest way of moving the checkers from the bottom left to the upper right. See Figure \ref{fig2.1}. Note that the serpent configuration has a speed of 1, since it requires only 2 moves to give it a forward progress of 2. \\

It was proved in \cite{GMT} that in a wide class of optimization problems on lattice structures (such as $\mathds{Z}^n$), the optimal way to move objects from one location to another  far away location, is to spend most of the time repeatedly translating very efficient configurations. 
In their 1993 paper \cite{Leapfrogging}, it was shown by Auslander, Benjamin, and Wilkerson that the maximum speed of any configuration is 1,  there are essentially only three configurations that achieve this speed, and they have at most four pieces. It was conjectured that with five or more pieces, the maximum attainable speed is 2/3. In this paper, we prove that conjecture.

\section{Abstracting the game}

Suppose we have $p$ indistinguishable pieces and wish to move them in the positive direction over the integer lattice $\mathds{Z}^n$. If a piece is located at coordinate $l \in \mathds{Z}^n$, and some other coordinate $l + e_i$ is not occupied by a piece (for unit vector $e_i$), then the piece may \textit{shift} there. Alternatively if $l + e_i$ is occupied but $l + 2e_i$ is not, the piece may \textit{hop} over the occupant of $l + e_i$ to land at $l + 2e_i$, where it may remain or continue hopping over other adjacent pieces. One legal \textit{move} consists of either a shift or a \textit{jump}, a sequence of one or more hops by a single piece.\\

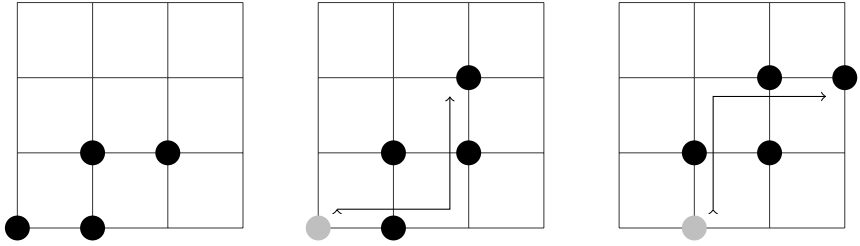
\begin{figure}[h]
    \centering
    \begin{tikzpicture}
        \gridThreeD{0}{0}{black!75}{3}{3}
        \gridThreeD{4}{0}{black!75}{3}{3}
        \gridThreeD{8}{0}{black!75}{3}{3}
        \node[] at (0,0) [circle, fill=black] {};
        \node[] at (1,0) [circle, fill=black] {};
        \node[] at (1,1) [circle, fill=black] {};
        \node[] at (2,1) [circle, fill=black] {};
        
        \node[] at (4,0) [circle, fill=black!25] {};
        \node[] at (5,0) [circle, fill=black] {};
        \node[] at (5,1) [circle, fill=black] {};
        \node[] at (6,1) [circle, fill=black] {};
        \node[] at (6,2) [circle, fill=black] {};
        \draw[->, to path={-| (\tikztotarget)}]
            (4.25,0.25) edge (5.75,1.75);
        
        \node[] at (9,0) [circle, fill=black!25] {};
        \node[] at (9,1) [circle, fill=black] {};
        \node[] at (10,1) [circle, fill=black] {};
        \node[] at (10,2) [circle, fill=black] {};
        \node[] at (11,2) [circle, fill=black] {};
        \draw[->, to path={|- (\tikztotarget)}]
            (9.25,0.25) edge (10.75,1.75);
    \end{tikzpicture}
        
    \caption{An example of the serpent's $2$-move trajectory, consisting of 2 jumps. The placements in the leftmost and rightmost diagrams are translates, and represented by the same configuration, the serpent. These two placements have displacement 2, and require 2 moves to reach one from the other. Hence, the serpent is a speed of light configuration, i.e. it has speed 1. }
    \label{fig2.1}
\end{figure}

Define a \textit{placement} of size $p$ as a finite subset of $\mathds{Z}^n$, denoted by $X = \{\vec{x}_1, \dots , \vec{x}_p\}$. Define the \textit{centroid} of placement $X$ to be
\[c(X) = \frac{1}{p}\sum_{u=1}^p \vec{x}_u.\]

For placements $X, Y$, define their \textit{displacement} as \[d(X,Y) = \sum_{i=1}^n |c_i(X) - c_i(Y)|.\]

For $m \geq 1$, an \textit{m-move trajectory} $X_0, X_1, ..., X_m$ is a sequence of placements where $X_{u+1}$ is reachable from $X_u$ in a single legal move. The \textit{speed} of an $m$-move trajectory from $X_0$ to $X_m$ is \[s = \frac{d(X_0,X_m)}{m}.\]

We say that placements $X, Y$ are \textit{translates} if there exists $\vec{a} \in \mathds{Z}^n$ such that $X + \vec{a} = Y$. We say that translates $X$ and $Y$ are represented by the same \textit{configuration} of pieces, and define the \textit{speed} of a configuration $C$ to be the maximum speed attained by any trajectory between two translates represented by $C$. \\

Auslander, Benjamin, and Wilkerson proved in 1993 the following: the maximum speed of any configuration $C$ is 1, and that only three configurations (called \textit{speed of light} configurations) attain this speed in $\mathds{Z}^n$ for $n\geq1$.\cite{Leapfrogging} These configurations are: 

\begin{itemize}
    \item The \textbf{atom} \{$x$\} (if $p = 1$).
    \item The \textbf{frog} \{$x$, $x + e_i$\}, $1 \leq i \leq n$ (if $p = 2$).
    \item The \textbf{serpent} \{$x$, $x + e_i$, $x + e_i + e_j$, $x + 2e_i + e_j$\} $1 \leq i \neq j \leq n$ (if $p = 4$ and $n >1$).
\end{itemize}

It was conjectured that the maximum attainable speed for any configuration on $p \neq 1,2,4$ is $2/3$, which we may observe is attained by the snake configuration with any number of pieces.\cite{Leapfrogging} We will show that in  $\mathds{Z}^2$, aside from these speed of light configurations,  the maximum achievable speed is $2/3$.

\section{Definitions and Properties}

Let $m \in \mathds{Z}$ and placement $X \in \mathds{Z}^n$. Then \textit{border} $l_m$ is defined by: \[l_m = \{x \in X: ||x||=m\}.\]

For a placement $X$, we may define the \textit{tail} (respectively, \textit{head}) of $X$, by $t(X) = \min _u$ $|l_u| \geq 0$ (respectively, $h(X) = \max _u$ $|l_u| \geq 0$). Define the \textit{width} of a placement $X$ $w(X) = h(X) - t(X) + 1$. Define the \textit{back border} (respectively, \textit{front border}) of $X$ as $T(X) := l_{t(X)}$ (respectively, $H(X) := l_{h(X)}$).\\ For example, if in the first diagram in Figure \ref{fig2.1}, the lower left piece is at $(0,0)$, then its configuration has a tail of 0 and a head of 3. 

We now define an underlying configuration which reoccurs in optimal play. A \textit{ladder} of length $k > 0$ is subset of a placement $X$: $L = \{p_0, p_1, ..., p_k\} \subseteq{X}$ such that $p_0$ is able to hop over $p_1,\dots, p_k$ successively. If $\{p_0\} = T(X)$ and $p_k \in H(X)$, then we say $L$ is a \textit{true ladder} of $X$. We call the move consisting of $p_0$ jumping over the rest of the ladder pieces a \textit{climb}, call $p_0$ the \textit{base} of the ladder, and the other pieces the \textit{rungs}. 

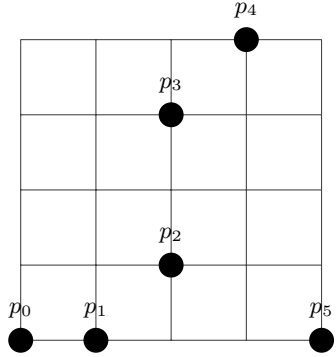
\begin{figure}[h]
    \centering
    \begin{tikzpicture}
        \gridThreeD{0}{0}{black!75}{4}{4}
        \node[label={\small $p_0$}] at (0,0) [circle, fill=black] {};
        \node[label={\small $p_1$}] at (1,0) [circle, fill=black] {};
        \node[label={\small $p_2$}] at (2,1) [circle, fill=black] {};
        \node[label={\small $p_3$}] at (2,3) [circle, fill=black] {};
        \node[label={\small $p_4$}] at (3,4) [circle, fill=black] {};
        \node[label={\small $p_5$}] at (4,0) [circle, fill=black] {};
    \end{tikzpicture}
        
    \caption{An example of a placement $X$ which contains a ladder. The ladder is $\{p_0,p_1, p_2, p_3, p_4\}$, with $p_0$ is the base of the ladder, and  pieces $p_1, p_2, p_3, p_4$ are the rungs. Here $\{p_0\} = T(X)$ and $\{p_4\} = H(X)$, so the ladder is a true ladder of $X$. $X$ has width 8.} 
    \label{fig3.1}
\end{figure}

\begin{proposition}
    If a configuration $X$ contains a true ladder $L$, $X$ has even width.
\end{proposition}

\begin{proof}
    Observe that when a piece $p$ hops over another piece $p'$, then $p$ either increases or decreases its border by 2. Therefore since $l_0$ jumps from $B_{t(X)}$ to $B_{h(X) + 1}$, this implies $h(X) + 1 - t(X)$ is even. 
\end{proof}

\begin{proposition}
    A placement $X$ with more than one piece can perform a move that simultaneously increases $t(X)$ and $h(X)$ if and only if it has a true ladder.
\end{proposition}

\begin{proof}
    $(\impliedby)$Performing a ladder climb increases both $t(X)$ and $h(X)$, as $l_0$ moves from $T(X)$, leaving that border empty, and jumps in front of $l_k \in H(X)$, thus advancing the front border. \\
    
    $(\implies)$ If a move on $X$ exists that advances the front and back borders forward, since only one piece can change positions, the back border must only have one piece, and it must be the piece which moves. Call this piece $p$. Since $X$ has more than one piece,  $w(X) > 1$, so $p$ must jump from $T(X)$ to in front of $H(X)$. Denote the sequence of pieces hopped over by $p$ by $p_1$, $p_2$, ... $p_k$. Since $p_k \in H(X)$, $\{p_0, p_1, ... , p_k\}$ is a true ladder. 
\end{proof}

We may classify possible moves into seven categories.

\begin{itemize}
    \item \textit{Ascent}: A move that increases $h(X)$ and $t(X)$. This is necessarily a ladder climb. 
    \item \textit{Front Push}: A move that increases $h(X)$ but not $t(X)$.
    \item \textit{Back Push}: A move that increases $t(X)$ but not $h(X)$.
    \item \textit{Dead Move}: A move that changes neither the tail nor head of $X$.
    \item \textit{Front Retreat}: A move that decreases the head of $X$.
    \item \textit{Back Retreat}: A move that decreases the tail of $X$.
    \item \textit{Reverse Ascent}: A move that decreases both the head and the tail of $X$.
\end{itemize}

For example, in the placement in Figure \ref{fig3.1}, $p_0$ climbing the ladder would be an ascent. $p_4$ shifting to the right would be a front push. $p_0$ jumping over $p_1$ and $p_2$ only would be a back push. $p_5$ shifting in any direction would be a dead move. \\

An ascent is necessarily a ladder climb for nontrivial placements. For a legal $m$-move trajectory $M = \{X_0, X_1, ..., X_m\}$, where $X_0$ is a translate of $X_M$, define the \textit{moveset} of $M$ as a collection of moves $ m(M) = \{x_0 \to x_0',$ ... $, x_{m-1} \to x_{m-1}'\}$, where $x_i$ is the location of the piece that moves in $X_i$, and $x_i'$ is the location of the moved piece in $X_{i+1}$.\\

\begin{proposition} 
    In any moveset, the total number of front pushes, back pushes, front retreats and back retreats which occur between two ascents must be even.
\end{proposition}
\begin{proof}
    Since ascents can only occur when a configuration has even width, the configurations immediately before and after any ascents must have even width. Therefore, the number of moves between two ascents which change the width parity must be even. The four listed move types are the only move types which change the parity of the width of a configuration, so the result follows. 
\end{proof}

For a move trajectory $M$, let $A(M)$ represent the number of ascents in $m(M)$, $FP(M)$ represent the number of front pushes, $BP(M)$ the number of back pushes, $DM(M)$ the number of dead moves, $FR(M)$ the number of front retreats, $BR(M)$ the number of back retreats, and $RA(M)$ the number of reverse ascents. \\

Now, define the \textit{weight} $\omega(M)$ of a trajectory $M$ or its corresponding moveset $m(M)$ as follows:
\[\omega(M)  := A(M) - (1/2) \cdot (FP(M) + BP(M)) - 2 \cdot DM(M) - (7/2) \cdot (FR(M) + BR(M)) - 5 \cdot RA(M) \]
    
This definition rewards (with higher weight) trajectories that make many efficient moves (e.g. ascents) and penalizes trajectories that use moves that make little forward progress or worse. For example, consider the trajectory where we begin with a snake with 5 pieces, located at points $\{(0,0), (1,1), (2, 2), (3, 3), (4, 4)\}$,  and in 3 moves (consisting of one back push, one ascent, and one front push), translate it to the points $\{(1,1), (2, 2), (3, 3), (4, 4), (5,5)\}$. Such a trajectory, with speed 2/3, would have weight $1 - (1/2)(1 + 1) = 0.$ 
Additionally, call the coefficient corresponding to each move type the \textit{move weight}. If a sequence of moves are performed, then the weight of the sequence is the sum of all the move weights. If we partition a trajectory $M = M_1 \oplus M_2 \oplus ... \oplus M_k$, then $\omega (M) = \omega(M_1) + ... + \omega(M_k)$. 

\begin{lemma}
    A $m$-move trajectory $M$ of a configuration $C$ has speed greater than $2/3$ if and only if $\omega(M) > 0$
\end{lemma}

\begin{proof}
    Since the move types are mutually exclusive, $A(M) + FP(M) + BP(M) + DM(M) + FR(M) + BR(M) + RA(M) = m$. Additionally, the displacement of $M$ can be characterized by $A(M) - RA(M) + (1/2)\times(FP(M) - FR(M) + BP(M) - BR(M))$. Therefore the speed of $M$ is
    
    \[\frac{A(M) - RA(M) + FP(M)/2 - FR(M)/2 + BP(M)/2 - BR(M)/2}{A(M) + FP(M) + BP(M) + DM(M) + FR(M) + BR(M) + RA(M)}.\]
    
    It is straightforward to check:
    \[2/3 < \frac{A(M) - RA(M) + FP(M)/2 - FR(M)/2 + BP(M)/2 - BR(M)/2}{A(M) + FP(M) + BP(M) + DM(M) + FR(M) + BR(M) + RA(M)} \iff 0 < \omega(M).\]
\end{proof}

We next introduce an important theorem which, as a corollary, demonstrates that most movesets do not have speed greater than 2/3.

\begin{theorem}
    For any trajectory $M$ with no two ascents occurring in a row and not both beginning and ending with an ascent, $\omega(M) \leq 0$.
\end{theorem}

\begin{proof}
    After an ascent, if a front push, back push, front retreat, or back retreat occurs, this changes the width of the placement, and since prior to the first move the placement had even parity, for another ascent to occur, another one of the listed moves must occur. \\
    
    Without loss of generality let us assume $M$ begins with an ascent, therefore it cannot end with one. Let us partition the moveset $m(M)$ into separate blocks $B_1,\dots, B_k$ where each block begins with an ascent and contains no other ascents. Since no two ascents can occur in a row, each block must consist of at least two moves. Additionally since each new block must begin with an ascent, a block must end with a placement with even width. \\
    
    $B_i$ contains one ascent and at least one other move $a$. If $a$ is any type of move besides a front or back push, $\omega(B_i)$ is negative. Otherwise, if $a$ is a front push or back push the resulting placement has odd width and another front push, back push, front retreat, or back retreat must occur. This implies $\omega(B_i) \leq 0$, with equality holding only when $B_i$ is of the form $\{A, FP/BP, FP/BP\}$. Since $B_i \leq 0$ for all $1 \leq i \leq k$, $\omega(M) \leq 0$, as desired. 
    
\end{proof}

Since we know that the maximum speed of a configuration is 1, and we are looking at non-maximal configurations, every moveset has at least one non-ascent move. Therefore, we will assume for the remainder of this paper that the last move in any moveset is not an ascent. 

\section{$p=1,2,4$}

For a configuration which is not speed of light, observe that it may reach speed arbitrarily close to speed 1 by first shifting into a speed-of-light configuration, repeating its set of moves sufficiently long, then moving back to the original configuration. We will consider only trajectories which do not use this strategy. Rigorously, if a trajectory between translates does not perform any moves corresponding to any ``speed-of-light'' configuration's optimal moveset, call the corresponding configuration \textit{non-speed-of-light}.

\begin{theorem}
	A non-speed-of-light configuration $C$ with 1, 2, or 4 pieces cannot have speed greater than 2/3.
\end{theorem}

\begin{proof}
When $p = 1$, this is immediately true, since there is only one configuration, and it can translate itself with speed 1 or - 1. 

For $p=2$, any trajectory containing a jump (which must be an ascent) must contain the frog's optimal trajectory (a jump), therefore any non-speed-of-light configuration is limited to only shifts. Therefore, the optimal trajectory of $C$ cannot contain any ascents, thus $C$ cannot have speed greater than 2/3.\\

For $p=4$, if $C$ has optimal trajectory without two ascents in a row, $C$ has speed at most 2/3. Suppose then that $C$ has two ascents in a row, without loss of generality let us assume its trajectory begins with two ascents. $C$ must have even width, and every border must contain a piece due to the parity of the ends of each successive ladder. Since $C$ can perform two ascents in a row, it has width 4. $C = \{p_1, p_2, p_3, p_4\}$ with $p_i$ on $l_i$. Say $p_1$ has jump $p_1: \xrightarrow{p_2} a_1 \xrightarrow{p_4} a_2$ for open locations $a_1, a_2$. $d(p_1, p_2)=1$, $d(a_1, p_2)=1$, and $d(a_1, p_4)=1$. Similarly, write the next move by $p_2$ as $p_2: \xrightarrow{p_3} b_1 \xrightarrow{p_1=a_2} b_2$. $d(p_2, p_3) = 1$ $d(p_3,b_1)=1$, and $d(a_2, b_1) = 1$. \\

Suppose $p_1$, $p_2$, and $p_4$ are collinear ($p_1$ on $x$, $p_2$ on $x + u_i$, $p_4$ on $x + 3u_i$), $a_2$ = $x + 4u_i$ and $a_1 = x + 2u_i$. However for $p_2$ to jump over $p_3$ and $p_1$, $p_3$ must occupy $a_1$, contradicting the assumption that $a_1$ is open. Say $p_1$ starts on $x$, $p_2$ on $x + u_i$, then $a_1$ is $x + 2u_i$, $p_4$ starts on $x + 2u_i + u_j$, and $a_2$ is  $x + 2u_i + 2u_j$. There are only two locations both adjacent to $p_2$ and 2 away from $a_2$, $a_1$ and  $x + u_i + u_j$. However if $p_3$ is on  $x + u_i + u_j$ and we perform the two ascents, we have performed the serpent configuration's trajectory, a contradiction. These cases are visualized in Figure \ref{fig4.1} Thus it is impossible for a non-speed-of-light configuration $C$ to have two ascents in a row, so $C$ must have speed at most 2/3. 
\end{proof}

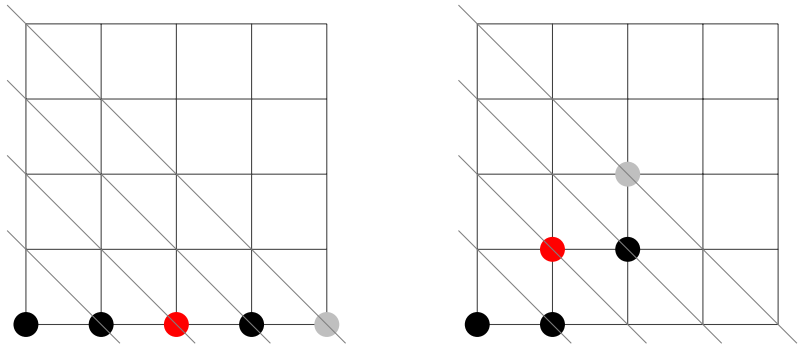
\begin{figure}[h]
    \centering
    \begin{tikzpicture}
        \gridThreeD{0}{0}{black!75}{4}{4}
        \node at (0,0) [circle, fill=black] {};
        \node at (1,0) [circle, fill=black] {};
        \node at (2,0) [circle, fill=red] {};
        \node at (3,0) [circle, fill=black] {};
        \node at (4,0) [circle, fill=black!25] {};
        \draw[black!50, thin](4.25,-.25) -- ++ (-4.5,4.5);
        \draw[black!50, thin](1.25,-.25) -- ++ (-1.5,1.5);
        \draw[black!50, thin](2.25,-.25) -- ++ (-2.5,2.5);
        \draw[black!50, thin](3.25,-.25) -- ++ (-3.5,3.5);
                    
        \gridThreeD{6}{0}{black!75}{4}{4}
        \node at (2+6,2) [circle, fill=black!25] {};
        \node at (1+6,0) [circle, fill=black] {};
        \node at (0+6,0) [circle, fill=black] {};
        \node at (2+6,1) [circle, fill=black] {};
        \node at (1+6,1) [circle, fill=red] {};
        \draw[black!50, thin](1.25+6,-.25) -- ++ (-1.5,1.5);
        \draw[black!50, thin](2.25+6,-.25) -- ++ (-2.5,2.5);
        \draw[black!50, thin](3.25+6,-.25) -- ++ (-3.5,3.5);
        \draw[black!50, thin](4.25+6,-.25) -- ++ (-4.5,4.5);
    \end{tikzpicture}
    
    \caption{In the case of $p=4$, we cannot place $p_3$ without a contradiction.}
    \label{fig4.1}
\end{figure}

\section{$p=3$}

\begin{theorem}
    No configuration of 3 pieces $C$ exists with speed greater than 2/3.
\end{theorem}

\begin{proof}
    To show this, we will demonstrate that no placement $X$ exists for $p=3$ such that two successive ascents are possible. If $X$ has two or three pieces occupying the same border, then the back or front border has two pieces, rendering successive ascents impossible. Otherwise assume $X$ has pieces occupying all different borders. The only possible way for $X$ to be able to perform a ladder climb is if it has width 4, since a piece jumping over two other pieces can travel distance at most 4. Let us consider the four borders passing through $X$, without loss of generality say $l_1 - l_4$. $l_{1}$ and $l_{4}$ must contain one piece each, $p_1$ and $p_3$ respectively, implying the last piece, $p_2$ can either lay on $l_2$ or $l_3$. If $p_2$ lies on $l_3$, $p_1$ cannot jump. Otherwise, $p_2$ lays on $l_2$. If $p_1$ can perform an ascent, the pieces now lay on $l_2$, $l_4$, and $l_5$, which implies $p_2$ cannot jump, as in Figure \ref{fig5.1}.
    
    \begin{figure}
        \centering
        \begin{tikzpicture}
            \gridThreeD{0}{0}{black!75}{4}{4}
            \node at (0,0) [circle, fill=black] {};
            \node at (1,0) [circle, fill=black] {};
            \node at (2,1) [circle, fill=black] {};
            \draw[black!50, thin](4.25,-.25) -- ++ (-4.5,4.5);
            \draw[black!50, thin](1.25,-.25) -- ++ (-1.5,1.5);
            \draw[black!50, thin](2.25,-.25) -- ++ (-2.5,2.5);
            \draw[black!50, thin](3.25,-.25) -- ++ (-3.5,3.5);
            
            \node at (5,2) {$\xrightarrow{}$};
            
            \gridThreeD{6}{0}{black!75}{4}{4}
            \node at (2+6,2) [circle, fill=black] {};
            \node at (1+6,0) [circle, fill=red] {};
            \node at (2+6,1) [circle, fill=black] {};
            \draw[black!50, thin](1.25+6,-.25) -- ++ (-1.5,1.5);
            \draw[red!50, thin](2.25+6,-.25) -- ++ (-2.5,2.5);
            \draw[black!50, thin](3.25+6,-.25) -- ++ (-3.5,3.5);
            \draw[black!50, thin](4.25+6,-.25) -- ++ (-4.5,4.5);
        \end{tikzpicture}
        
        \caption{After a jump, $p$ is isolated and cannot jump.}
        \label{fig5.1}
    \end{figure}
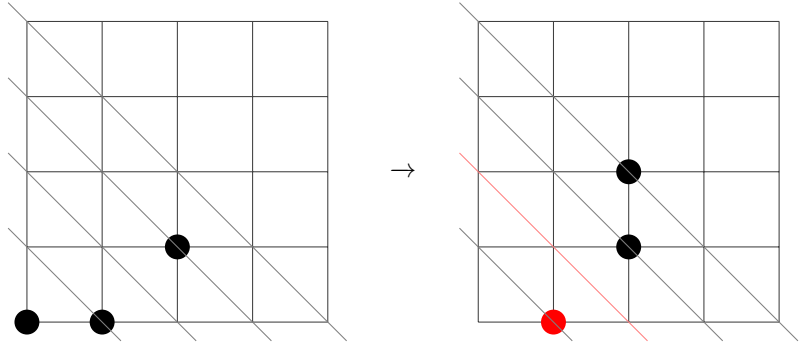
    Therefore no placement $X$ with $p=3$ pieces exists such that two consecutive ascents can be performed, as desired. No such configuration of $\mathds{Z}^n$ exists with speed greater than $2/3$.
\end{proof}

\section{$p>4$ for $\mathbb{Z}^2$}

We provide a proof that no configuration of greater than 4 pieces has a speed greater than 2/3 in the 2-dimensional case. We are confident that the approach taken here will produce the same sort of result in $\mathbb{Z}^n$, but we shall not pursue that here. Recall that if a moveset has no consecutive ascents, by Theorem 3.2, the moveset has speed at most 2/3, so our focus will be on movesets that have consecutive ascents. 

\begin{lemma}
    For $p>4$, there does not exist a configuration with a moveset containing 4 or more consecutive ascents. 
\end{lemma}
\begin{proof}
    Suppose for contradiction that there exists a configuration $C$ with moveset containing more than 3 consecutive ascents. Then there can only be one piece on each of the four backmost borders. Additionally, the width of $C$ prior to moving must be at least 6. It follows that without loss of generality, we can assume the first four pieces which ascend are located at $(0,0), (1,0), (1,1), \&\, (2,1)$, as in Figure \ref{fig6.1}.
    \begin{figure}[h]
        \centering
        \begin{tikzpicture}
            \gridThreeD{0}{0}{black!75}{4}{5}
            \gridThreeD{6}{0}{black!75}{4}{5}
            \node[label={\small $p_1$}] at (0,0) [circle, fill=black] {};
            \node[label={\small $p_2$}] at (1,0) [circle, fill=black] {};
            \node[label={\small $p_3$}] at (1,1) [circle, fill=black] {};
            \node[label={\small $p_4$}] at (2,1) [circle, fill=black] {};
            
            \node at (8,1) [circle, fill=black] {};
            
            \node at (2,0) {$\times$};
            \node at (1,2) {$\times$};
            \node at (2,2) {$\times$};
            \node at (3,1) {$\times$};
            \node at (5,2.5) {$\ldots$};

            \node at (6,0) {$\times$};
            \node at (7,0) {$\times$};
            \node at (7,1) {$\times$};
            \node at (7,2) {$\times$};
            \node at (8,0) {$\times$};
            \node at (9,1) {$\times$};
            \node at (8,2) {$\times$};
    
            \draw[black!50, thin](3.25,-.25) -- ++ (-3.5,3.5);
            \draw[red!50, thin](4.25,-.25) -- ++ (-4.5,4.5);
            \draw[black!50, thin](1.25,-.25) -- ++ (-1.5,1.5);
            \draw[black!50, thin](2.25,-.25) -- ++ (-2.5,2.5);
            
        \end{tikzpicture}
        \caption{The back borders of $C$ before any moves on the left, and after the first three ascents on the right. An $\times$ indicates a location a piece cannot be located in that board state. }
        \label{fig6.1}
    \end{figure}
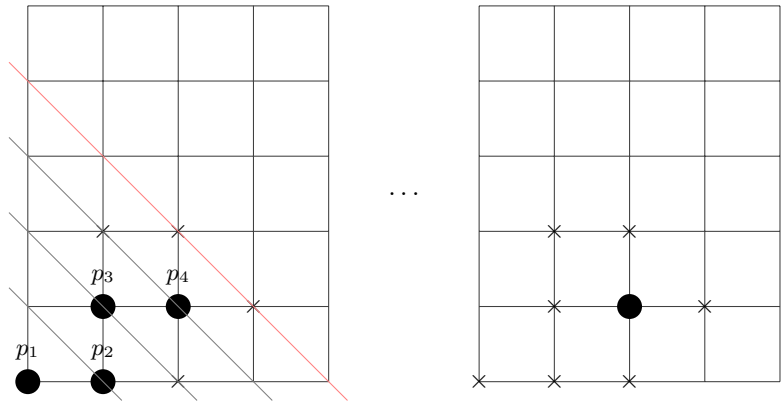
    
    However, after the first 3 ascents, the piece at $(2,1)$ $p_4$ is not adjacent to any pieces and therefore cannot ascend, a contradiction. 
\end{proof}

Define a piece's \textit{measure} by taking its location modulo 2, $\overbar{(x,y)}_2$. For example, $p_4$ in the above diagram has measure $(0,1)$. Note that when a piece jumps, its measure stays constant. This restricts the number of locations in $\mathds{Z}^2$ a piece can jump to, given its starting position. \\

We assume without loss of generality that a moveset $M$ begins with the maximum number of ascents. Note that this implies $M$ ends in a non-ascent. Define an \textit{isolating partition} of $M$ as follows. First, partition $m(M)$ sequentially into blocks $A_1, ... A_k$ such that each block begins with two or more consecutive ascents, but does not have consecutive ascents anywhere else and does not end with an ascent. So each new block begins at every occurrence of a sequence of two or more consecutive ascents, and ends with a non-ascent. This partition of $M$ is unique. 
For example, if a configuration had moveset of type $\{A, A, A, FP, DM, BP, A, A, DM, DM, A, FP, BP\}$, then the moveset would be partitioned into blocks $A_1 = \{A,A,A, FP, DM, BP\}$, $A_2 = \{A, A, DM, DM, A, FP, BP\}$. 

Let $L(A_i)$ be the number of ascents $A_i$ begins with. In our example, $L(A_1) = 3$ and $L(A_2) = 2$. By Lemma 6.1, $L(A_i) \leq 3$. Since $A_i$ ends with a non-ascent,  $\omega(A_i) < L(A_i)$. We wish to show $\omega(A_i) \leq 0$ for all $i$, since $\omega(M) = \omega(\sum A_i) = \sum \omega(A_i)$. Hence if $\omega(A_i) \leq 0$ for all $i$, then $\omega(M) \leq 0$. So it suffices to only consider blocks rather than entire movesets. Since $A_i$ can only begin with 2 or 3 ladder climbs, we only have these two cases to consider. \\

Call a block $A_i$ for which $\omega(A_i) \leq 0$ \textit{suboptimal}. Recall that  

\begin{lemma}
    If a consecutive sequence of moves $S \subset A_i$ satisfies $\omega(S) + L(A_i) \leq 0$, then necessarily, $\omega(A_i) \leq 0$, and thus $A_i$ is suboptimal.
\end{lemma}
\begin{proof}
    We may assume without loss of generality that $S$ does not begin with an ascent, since if it does, then removing the initial ascents produces another consecutive sequence of moves satisfying the same condition. We may also assume that the move following $S$ is an ascent or the end of the block, since otherwise, we may extend $S$ until the move following it is an ascent. \\
    
    Now, partition $A_i$ into smaller blocks in the following way: let the first block contain all consecutive initial ascents.. Then, until the beginning of $S$ is reached, partition the moves so each block begins with a non-ascent and continues until an ascent is reached, and let the block end with an ascent. Each of these blocks will contain exactly one ascent and contain at least one non-ascent, since $A_i$ only contains one sequence of consecutive ascents. By assumption, the last move before $S$ begins is an ascent, so this partitioning continues until $S$ is reached. Let $S$ be the next block. Then, partition the remaining moves of $A_i$ in the same way as in the proof of Theorem 3.2, that is, by letting each block begin with an ascent and continuing until another ascent is reached. These blocks will also contain only one ascent. \\
    
    By the same arguments used in the proof of Theorem 3.2, all blocks $B$ which are not the first block and $S$ satisfy $\omega(B) \leq 0$. Since the only blocks remaining are $S$ and the first one, which has weight $L(A_i)$, by summing the weights of each block, we conclude that $\omega(A_i) \leq 0$. 
\end{proof}

Call $S$ a \textit{suboptimal sequence of moves}, and call any consecutive sequence of moves $S$ that is not suboptimal \textit{optimal}. For example, if 2 ascents were followed by 4 front pushes, the four front pushes would be a suboptimal sequence of moves, and the block $A_i$ would therefore be suboptimal. Our strategy is to show that any block $A_i$ must contain a suboptimal sequence of moves.


\begin{lemma}
    If a block $A_i$ begins with exactly three ascents, $\omega(A_i) \leq 0$.
\end{lemma}
\begin{proof}
    Observe that if $A_i$ begins with exactly three consecutive ascents, the initial placement is forced to have a serpent configuration at the back, and after the three ascents, the resulting configuration is forced to have a serpent configuration in the front, as demonstrated in Figure \ref{fig6.2}. \\
    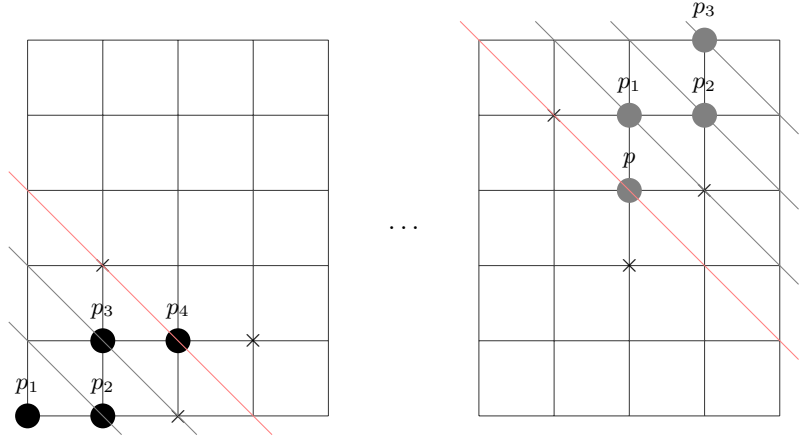
\begin{figure}[h]
        \centering
        \begin{tikzpicture}
            \gridThreeD{0}{0}{black!75}{4}{5}
            \gridThreeD{6}{0}{black!75}{4}{5}
            \node[label={\small $p_1$}] at (0,0) [circle, fill=black] {};
            \node[label={\small $p_2$}] at (1,0) [circle, fill=black] {};
            \node[label={\small $p_3$}] at (1,1) [circle, fill=black] {};
            \node[label={\small $p_4$}] at (2,1) [circle, fill=black] {};
            
            \node[label={\small $p$}] at (8,3) [circle, fill=black!50] {};
            \node[label={\small $p_1$}] at (8,4) [circle, fill=black!50] {};
            \node[label={\small $p_2$}] at (9,4) [circle, fill=black!50] {};
            \node[label={\small $p_3$}] at (9,5) [circle, fill=black!50] {};
            
            \node at (3,1) {$\times$};
            \node at (5,2.5) {$\ldots$};
            \node at (7,4) {$\times$};
            \node at (8,2) {$\times$};
            \node at (9,3) {$\times$};
            \node at (1,2) {$\times$};
            \node at (2,0) {$\times$};
    
            \draw[red!50, thin](3.25,-.25) -- ++ (-3.5,3.5);
            \draw[black!50, thin](1.25,-.25) -- ++ (-1.5,1.5);
            \draw[black!50, thin](2.25,-.25) -- ++ (-2.5,2.5);
            
            \draw[red!50, thin](10.25,0.75) -- ++ (-4.5,4.5);
            \draw[black!50, thin](10.25,1.75) -- ++ (-3.5,3.5);
            \draw[black!50, thin](10.25,2.75) -- ++ (-2.5,2.5);
            \draw[black!50, thin](10.25,3.75) -- ++ (-1.5,1.5);
    
        \end{tikzpicture}
        
        \caption{The borders before and after the first three ascents. The red borders denote  forced placements. An $\times$ indicates a location a piece cannot be located in that board state. }
        \label{fig6.2}
    \end{figure}
    
    By considering move weights, it suffices to demonstrate that a consecutive sequence of moves in $A_i$ occurs with weight $-3$, that is, a suboptimal sequence of moves must occur. Since each non-ascent has weight at most $-1/2$ if six non-ascents occur between two ascents, then the condition is satisfied. Note that if five non-ascents occur between two ascents, a sixth must occur by Proposition 3.0.3. Alternatively, if two dead moves occur or one dead move and one front/back push occurs, the condition is also satisfied for the same reasoning. Finally, any front retreat, back retreat, or reverse ascent is immediately suboptimal. \\
    
    Call the border containing the backmost piece prior to moving $l_1$. We first consider two cases, if $l_4$ has exactly 1 piece or if $l_4$ has two or more pieces. \\
    
    If $l_4$ has two or more pieces, first suppose the next ascent occurs from $l_4$. This implies that one of the pieces on $l_4$ must front push or dead move from $l_4$ first. Observe that a ladder climb to the front from $l_4$ cannot be performed unless $p$ or $p_2$ moves, since the pieces from $l_4$ can only hop over $p_1$ and $p_3$ by parity. If a piece from $l_4$ dead moves, then since an ascent can still not occur, another move must be made before an ascent is possible, so since a dead move and one more non-ascent has been performed, there is a suboptimal sequence of moves.\\
    
    Therefore, a possibly optimal sequence of moves must have a piece from $l_4$ make a front push before the other performs the ascent. A front push consists of a ladder climb from $l_4$ to beyond the front border, but one of $p$ or $p_2$ must move before that is possible. If $p$ or $p_2$ performs a dead move, then after the ladder climb from $l_4$, a suboptimal sequence of moves will be forced to occur before the next ascent. \\
    
    If $p$ or $p_2$ ladder climbs to perform a front push, a ladder climb from $l_4$ still cannot be performed, so necessarily, a front push must occur by the piece that previously moved. After this, it may be possible for a piece, call it $p'$ from $l_4$ to front push. The only optimal move afterwards is another front push, and since the remaining piece on $l_4$ can only hop over $p_1$ and $p_3$ in the front borders, for an ascent to be possible, $p'$ must front push. If an ascent is not possible, then since four front pushes have occurred and another non-ascent must occur, a suboptimal sequence of moves must happen. Otherwise, suppose an ascent from $l_4$ happens. The sequence of moves performed is $A,A,A, FP, FP, FP, FP, A$, and the sequence of moves after the consecutive ascents has weight -1. Now, if a new sequence of moves of weight -2 occurs, the total sequence will be suboptimal.\\
    
    From the initial configuration, Figure \ref{fig6.2}, it is necessary that the pieces on $l_4$ were either farther than distance 2 away, or were distance 2 away, but must have hopped over separate pieces. Thus, there are 2 pieces on $l_5$. If the next ascent occurs from $l_5$, the border beyond $p_3$ is empty and requires a piece there to build a ladder. This border cannot be filled by a back push since $l_5$ has 2 pieces, and thus, a dead move must occur before the ascent, forcing a suboptimal sequence of moves. Otherwise, if the next ascent occurs from beyond $l_5$, since a ladder climb from $l_5$ to the front is impossible, a dead move must occur when bringing the pieces on $l_5$ forward, again, forcing suboptimality. We conclude that if $l_4$ began with 2 or more pieces, and the first ascent after the three initial ones occurred from $l_4$, there is a suboptimal sequence of moves. \\
    
    Otherwise, suppose $l_4$ contains at least 2 pieces, and the first ascent after the initial three happens beyond $l_4$. In this case $l_4$ must be cleared, but neither piece can front push (a ladder climb to in front of $p_3$). If one piece dead moves forward and the other front pushes, then a suboptimal sequence of moves occurs. Otherwise, if no dead move occurs, then necessarily, one of $p$ or $p_2$ front pushes twice to open the ladder, then a piece front pushes from $l_4$, and another back pushes to clear $l_4$. However by a previous argument, there must be at least 2 pieces on $l_5$, so another move must occur before an ascent can. Since 5 moves have occurred, suboptimality is forced. Thus, if $l_4$ began with at least 2 pieces, a suboptimal sequence of moves must occur. \\

    Now, suppose $l_4$ has one piece, and let $p$ be the original front border piece. The back of the initial configuration must necessarily must be in the situation in Figure \ref{fig6.3}:\\
    
    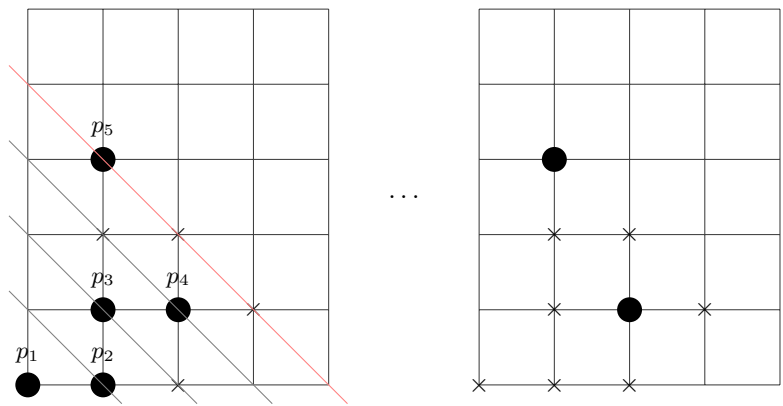
\begin{figure}[!h]
        \centering
        \begin{tikzpicture}
            \gridThreeD{0}{0}{black!75}{4}{5}
            \gridThreeD{6}{0}{black!75}{4}{5}
            \node[label={\small $p_1$}] at (0,0) [circle, fill=black] {};
            \node[label={\small $p_2$}] at (1,0) [circle, fill=black] {};
            \node[label={\small $p_3$}] at (1,1) [circle, fill=black] {};
            \node[label={\small $p_4$}] at (2,1) [circle, fill=black] {};
            \node[label={\small $p_5$}] at(1,3)  [circle, fill=black] {};
            \node at (8,1) [circle, fill=black] {};
            \node at (7,3) [circle, fill=black] {};

            \node at (2,0) {$\times$};
            \node at (1,2) {$\times$};
            \node at (2,2) {$\times$};
            \node at (3,1) {$\times$};
            
            \node at (5,2.5) {$\ldots$};

            \node at (6,0) {$\times$};
            \node at (7,0) {$\times$};
            \node at (7,1) {$\times$};
            \node at (7,2) {$\times$};
            \node at (8,0) {$\times$};
            \node at (9,1) {$\times$};
            \node at (8,2) {$\times$};
    
            \draw[black!50, thin](3.25,-.25) -- ++ (-3.5,3.5);
            \draw[red!50, thin](4.25,-.25) -- ++ (-4.5,4.5);
            \draw[black!50, thin](1.25,-.25) -- ++ (-1.5,1.5);
            \draw[black!50, thin](2.25,-.25) -- ++ (-2.5,2.5);
            
        \end{tikzpicture}
        
        \caption{The left diagram is the initial configuration, and the right diagram is the configuration after the three consecutive ascents. }
        \label{fig6.3}
    \end{figure}
    
    If the next ascent occurs on $l_4$, then again, $p$ is forced to dead move, and since $p_4$ is isolated, another piece must move to be adjacent to it. If these are separate moves, then two dead moves have been performed, a sequence of moves with weight $-4$. Otherwise, suppose these are the same move, so $p_4$ ascends right after. Then there is a new serpent configuration at the front, and two pieces on $l_5$, namely $p_5$ and the piece $p_4$ hopped over to ascend (which is necessarily $p$ as in Figure \ref{fig6.2}), as pictured in Figure \ref{fig6.4}. \\
    
    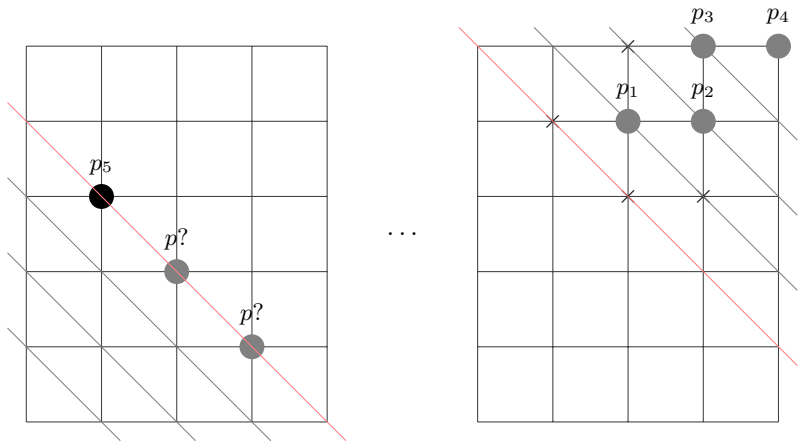
\begin{figure}[h]
        \centering
        
        \begin{tikzpicture}
            \gridThreeD{0}{0}{black!75}{4}{5}
            \gridThreeD{6}{0}{black!75}{4}{5}
            \node[label={\small $p$?}] at (2,2) [circle, fill=black!50] {};
            \node[label={\small $p$?}] at (3,1) [circle, fill=black!50] {};
            \node[label={\small $p_5$}] at(1,3)  [circle, fill=black] {};
            
            \node[label={\small $p_4$}] at (10,5) [circle, fill=black!50] {};
            \node[label={\small $p_1$}] at (8,4) [circle, fill=black!50] {};
            \node[label={\small $p_2$}] at (9,4) [circle, fill=black!50] {};
            \node[label={\small $p_3$}] at (9,5) [circle, fill=black!50] {};
            
            \node at (5,2.5) {$\ldots$};
            \node at (7,4) {$\times$};
            \node at (8,5) {$\times$};
            \node at (9,3) {$\times$};
            \node at (8,3) {$\times$};
            
            \draw[black!50, thin](3.25,-.25) -- ++ (-3.5,3.5);
            \draw[red!50, thin](4.25,-.25) -- ++ (-4.5,4.5);
            \draw[black!50, thin](1.25,-.25) -- ++ (-1.5,1.5);
            \draw[black!50, thin](2.25,-.25) -- ++ (-2.5,2.5);
            
            \draw[red!50, thin](10.25,0.75) -- ++ (-4.5,4.5);
            \draw[black!50, thin](10.25,1.75) -- ++ (-3.5,3.5);
            \draw[black!50, thin](10.25,2.75) -- ++ (-2.5,2.5);
            \draw[black!50, thin](10.25,3.75) -- ++ (-1.5,1.5);
    
        \end{tikzpicture}
        
        \caption{The configuration after $p_4$ ascends, in the case where $l_4$ has one piece and the next ascent follows from $l_4$. There are necessarily two pieces on $l_5$ and a serpent configuration in the front preventing an ascent from occurring. }      
        \label{fig6.4}
    \end{figure}    
    
    Since the previous two moves have weight $-1$, it is enough to show the next sequence of moves must have weight $-2$. Thus it suffices to assume no dead move can occur. Note that the parity is correct for an ascent to occur, as the pieces on $l_5$ can hop over $p_2$ and $p_4$, but given the front of the configuration, no ascent can occur. At least two front pushes (or a dead move) must occur in the front borders before one of the pieces on $l_5$ can climb to the front. However, this is three front pushes, implying a fourth front or back push must occur before an ascent can occur, so there must be a sequence of moves with weight $-2$ if the next ascent is to occur from $l_5$. \\
    
    Otherwise, if the next ascent happens beyond $l_5$, it is straightforward to see that a sequence of moves must first occur with weight $-3$, and we conclude that in any case, $\omega(A_i) \leq 0$.
\end{proof}

\begin{lemma}
    If a block $A_i$ begins with exactly two ascents, $\omega(A_i) \leq 0$. 
\end{lemma}
\begin{proof}
    By considering move weights, it suffices to demonstrate that a sequence of moves in $A_i$ after the two ascents occurs with weight -2, that is, a suboptimal sequence of moves must occur. In particular, it suffices to show either a dead move occurs or 4 front or back pushes occur between ascents. However, by Proposition 3.0.3, it suffices to show only 3 front or back pushes occur rather than 4. First observe that the starting and ending configurations after the two ascents must be as follows at the front and back borders respectively, pictured in Figure \ref{fig6.5}.\\
    
    \begin{figure}[h]
        \centering
        
        \begin{tikzpicture}
            \gridThreeD{0}{0}{black!75}{4}{5}
            \gridThreeD{6}{0}{black!75}{4}{5}
            \node[label={\small $p_1$}] at (0,0) [circle, fill=black] {};
            \node[label={\small $p_2$}] at (1,0) [circle, fill=black] {};
            \node at (1,1) [circle, fill=black] {};

            \node[label={\small $p$}] at (8,3) [circle, fill=black!50] {};
            \node[label={\small $p_1$}] at (8,4) [circle, fill=black!50] {};
            \node[label={\small $p_2$}] at (9,4) [circle, fill=black!50] {};

            \node at (5,2.5) {$\ldots$};
            \node at (7,4) {$\times$};
            \node at (8,2) {$\times$};
            \node at (1,2) {$\times$};
            \node at (2,0) {$\times$};
    
            \draw[black!50, thin](1.25,-.25) -- ++ (-1.5,1.5);
            \draw[red!50, thin](2.25,-.25) -- ++ (-2.5,2.5);
            
            \draw[red!50, thin](10.25,0.75) -- ++ (-4.5,4.5);
            \draw[black!50, thin](10.25,1.75) -- ++ (-3.5,3.5);
            \draw[black!50, thin](10.25,2.75) -- ++ (-2.5,2.5);
            \draw[black!50, thin](10.25,3.75) -- ++ (-1.5,1.5);
    
        \end{tikzpicture}
        
        \caption{The borders before/after the ones indicated by red in the above diagrams are forced placements. A $\times$ indicates a location a piece cannot be. }
        \label{fig6.5}
    \end{figure}
    
    Call the border containing the piece $p_1$ which is farthest back $l_1$. We begin by considering two cases, if $l_3$ has exactly 1 piece, $p_3$, or if it has multiple pieces $p_3$ and $p_3'$. Then, we consider sub-cases and determine that any trajectory must eventually become suboptimal. \\
    
    First, suppose $l_3$ has at least 2 pieces initially, $p_3$ and $p_3'$. Let us consider the next two moves following the ascents. If either is a dead move or worse, the trajectory is suboptimal, so suppose the next two moves are front or back pushes. Since $l_3$ has at least 2 pieces, the first move must be a front push. If the front push is not performed by $p_3$ or $p_3'$, then the next move must also be a front push, but since the width of the configuration is odd after the first front push, $p_3$ or $p_3'$ could not have performed the front push. Since $l_3$ contains two pieces still, the next move is not an ascent, so the sequence of moves is suboptimal.\\
    
    Otherwise, suppose the first front push was a ladder climb from $l_3$ to the front. Without loss of generality suppose $p_3$ makes the climb. If the second move is a back push from $l_3$, since there is a serpent configuration consisting of $p, p_1, p_2, p_3$, a ladder climb is not possible afterwards, and thus the trajectory is suboptimal. \\
    
    Instead, suppose the second move is a front push. Since there is a serpent in the front, the only move which allows for a ladder climb (necessarily by $p_3'$) is $p_3$ shifting forward. Then, a ladder climb by $p_3'$ must finish by hopping to where $p_3$ was prior to shifting, then hopping over $p_3$. This implies $p_3$ and $p_3'$ originally had the same measure, so they were distance at least 4 away. Therefore, there must be at least 2 pieces on $l_4$ which $p_3$ and $p_3'$ hopped over, $p_4$ and $p_4'$, as pictured in Figure \ref{fig6.6}. 
    
    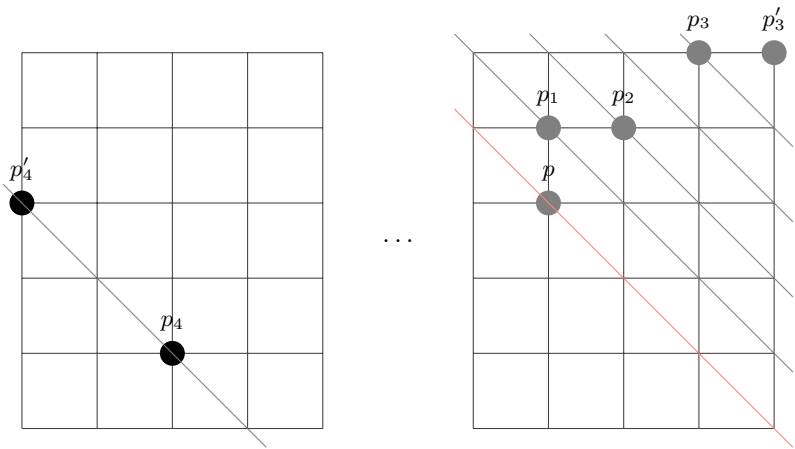
\begin{figure}[h]
        \centering
       
        \begin{tikzpicture}
            \gridThreeD{0}{0}{black!75}{4}{5}
            \gridThreeD{6}{0}{black!75}{4}{5}
            \node[label={\small $p_4'$}] at (0,3) [circle, fill=black] {};
            \node[label={\small $p_4$}] at (2,1) [circle, fill=black] {};

            \node[label={\small $p$}] at (7,3) [circle, fill=black!50] {};
            \node[label={\small $p_1$}] at (7,4) [circle, fill=black!50] {};
            \node[label={\small $p_2$}] at (8,4) [circle, fill=black!50] {};
            \node[label={\small $p_3$}] at (9,5) [circle, fill=black!50] {};
            \node[label={\small $p_3'$}] at (10,5) [circle, fill=black!50] {};

            \node at (5,2.5) {$\ldots$};

            \draw[red!50, thin](10.25,-0.25) -- ++ (-4.5,4.5);
            \draw[black!50, thin](10.25,0.75) -- ++ (-4.5,4.5);
            \draw[black!50, thin](10.25,1.75) -- ++ (-3.5,3.5);
            \draw[black!50, thin](10.25,2.75) -- ++ (-2.5,2.5);
            \draw[black!50, thin](10.25,3.75) -- ++ (-1.5,1.5);
            
            \draw[black!50, thin](3.25,-0.25) -- ++ (-3.5,3.5);
    
        \end{tikzpicture}
        
        \caption{Right: the only possibly optimal trajectory if 2 pieces started on $l_3$, five moves in. Left: a corresponding possible placement on $l_4$ five moves in.}
        \label{fig6.6}
    \end{figure}
    
    By a parity argument, $p_4$ and $p_4'$ cannot hop over $p, p_2,$ or $p_3$, as they are on an even-numbered border. If the next ladder climb occurs from beyond $l_4$, then $p_4, p_4',$ and $p_3'$ all must move first, which is necessarily a sequence of moves with weight at maximum -2. Otherwise, if the next ladder climb occurs on $l_4$, the pieces in front of $p$ which can be used in the ladder are $p_1$ and $p_3'$. However, a piece needs to move to the border between the ones containing $p_2$ and $p_3$, and by parity, this piece cannot come from $l_4$. Hence, this must be a dead move, and we conclude that if $l_3$ began with 2 or more pieces, then $\omega(A_i) \leq 0$. \\
    
    Otherwise, suppose $l_3$ began with exactly 1 piece on it. We consider subcases based on where the next ascent in the the trajectory occurs from and the number of pieces on the following borders. First, suppose the next ascent comes from $l_4$ or beyond. If $l_4$ contains more than one piece, then in order for the trajectory to be optimal, the next two moves must be a back push from $l_3$, then a front push from $l_4$. However, this is impossible, since a back push from $l_4$ must be a ladder climb, but by a parity argument of where the frontmost border is, a ladder climb from $l_4$ cannot advance the front border. \\
    
    Otherwise, suppose the next ascent comes from $l_4$ or beyond and that $l_4$ contains exactly one piece. The necessary starting configuration is pictured in Figure \ref{fig6.7}.     
    \begin{figure}[h]
        \centering
       
        \begin{tikzpicture}
            \gridThreeD{0}{0}{black!75}{4}{5}
            \gridThreeD{6}{0}{black!75}{4}{5}
            \node[label={\small $p_1$}] at (0,0) [circle, fill=black] {};
            \node[label={\small $p_2$}] at (1,0) [circle, fill=black] {};
            \node[label={\small $p_3$}] at (1,1) [circle, fill=black] {};
            \node[label={\small $p_4$}] at (2,1) [circle, fill=black] {};
            \node[label={\small $p_5$}] at (1,3) [circle, fill=black] {};

            \node[label={\small $p$}] at (8,3) [circle, fill=black!50] {};
            \node[label={\small $p_1$}] at (8,4) [circle, fill=black!50] {};
            \node[label={\small $p_2$}] at (9,4) [circle, fill=black!50] {};

            \node at (5,2.5) {$\ldots$};
            \node at (7,4) {$\times$};
            \node at (8,2) {$\times$};
            \node at (1,2) {$\times$};
            \node at (2,0) {$\times$};
    
            \draw[black!50, thin](1.25,-.25) -- ++ (-1.5,1.5);
            \draw[red!50, thin](2.25,-.25) -- ++ (-2.5,2.5);
            
            \draw[red!50, thin](10.25,0.75) -- ++ (-4.5,4.5);
            \draw[black!50, thin](10.25,1.75) -- ++ (-3.5,3.5);
            \draw[black!50, thin](10.25,2.75) -- ++ (-2.5,2.5);
            \draw[black!50, thin](10.25,3.75) -- ++ (-1.5,1.5);
    
        \end{tikzpicture}
        
        \caption{On the left, part of the necessary configuration if $l_4$ contains one piece, prior to the first two ascents. On the right, the front borders after the first two ascents. }
        \label{fig6.7}
    \end{figure}
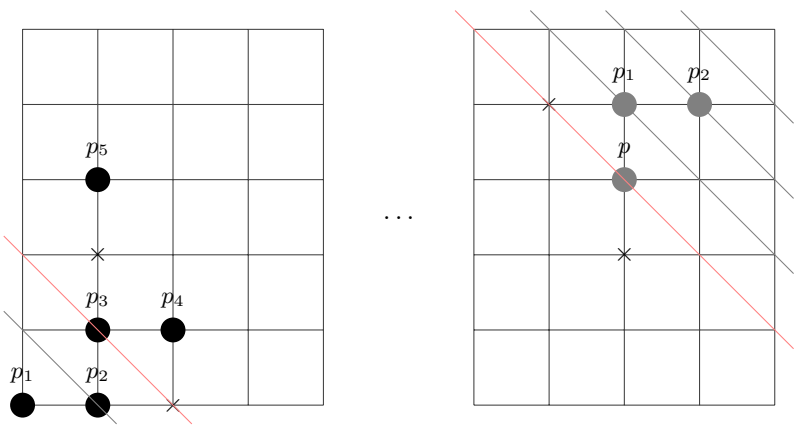
    
    Then for an optimal trajectory, the next two moves following the two ascents must be a back push from $l_3$ and a front push elsewhere (by assumption an ascent is not allowed). Assume the same piece does not perform the two moves, since otherwise, an ascent would be a faster trajectory, which would reduce to the previous lemma. Note that the front push cannot occur from $l_4$ by parity. Moreover, the back push from $l_3$ must hop over the sole piece on $l_4$, because otherwise, after the two moves there would be 2 pieces on $l_4$. 

    Finally, note that the front push must necessarily come from $p_2$ if an ascent is possible afterwards, since it is clear that a front push from  $p_1$ prevent ascents, and front pushes from any other piece result in a serpent configuration at the front which prevents an ascent. Hence, after the front and back push, if an ascent is possible, there must be two pieces on $l_5$: the piece which $p_2$ hopped over in its ascent, and the piece $l_4$ is to hop over first in its ascent (this may be $p_3$). After $p_4$ ascends, we now are in the configuration pictured in Figure \ref{fig6.8}.\\
    
    \begin{figure}[h]
        \centering
        
        \begin{tikzpicture}
            \gridThreeD{0}{0}{black!75}{4}{5}
            \gridThreeD{6}{0}{black!75}{5}{5}
            \node[label={\small $p$?}] at (2,2) [circle, fill=black!50] {};
            \node[label={\small $p_5$}] at (1,3)  [circle, fill=black] {};
            
            \node[label={\small $p$}] at (8,3) [circle, fill=black!50] {};
            \node[label={\small $p_1$}] at (8,4) [circle, fill=black!50] {};

            \node[label={\small $p_2?$}] at (10,4) [circle, fill=black!50] {};
            \node[label={\small $p_4?$}] at (11,4) [circle, fill=black!50] {};
            
            \node at (5,2.5) {$\ldots$};
            
            \draw[black!50, thin](3.25,-.25) -- ++ (-3.5,3.5);
            \draw[red!50, thin](4.25,-.25) -- ++ (-4.5,4.5);
            \draw[black!50, thin](1.25,-.25) -- ++ (-1.5,1.5);
            \draw[black!50, thin](2.25,-.25) -- ++ (-2.5,2.5);
            
            \draw[red!50, thin](11.25,-.25) -- ++ (-5.5,5.5);
            \draw[black!50, thin](11.25,0.75) -- ++ (-4.5,4.5);
            \draw[black!50, thin](11.25,1.75) -- ++ (-3.5,3.5);
            \draw[black!50, thin](11.25,2.75) -- ++ (-2.5,2.5);
    
        \end{tikzpicture}
        
        \caption{The configuration after $p_4$ ascends. The question marks denote where piece locations are not forced. }
        \label{fig6.8}
    \end{figure}
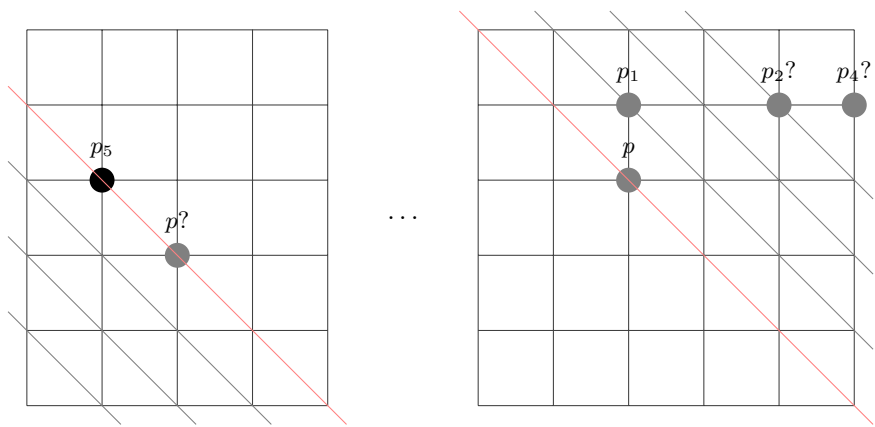 
    
    Now, for the trajectory to stay optimal, the next ascent must occur from $l_5$, since otherwise it would take two moves to clear $l_5$ and at least one to alter the front borders so that an ascent is possible. Pieces from $l_5$ may only hop over $p$ and $p_4$ in the front borders, so they must be in the next ladder. This implies that in the next two moves before the ascent, a piece must move to the border between $p_1$ and $p_2$. However, this must be a dead move or worse, since $l_5$ contains at least two pieces and thus cannot initiate a back push. Thus, all movesets for which the next ascent comes from $l_4$ or beyond are suboptimal. \\
    
    Finally, we consider the case where the next ascent comes from $l_3$, and $l_3$ contains only one piece. Since an ascent cannot immediately occur, the next two moves if they are to be optimal must be front pushes. It is quick to see that the only possibilities are either $p_1$ pushing twice or another unlabeled piece coming from an odd-numbered border pushing twice, as pictured in Figure \ref{fig6.9}.\\
    
    \begin{figure}[h]
        \centering
        
        \begin{tikzpicture}
            \gridThreeD{0}{0}{black!75}{4}{5}
            \gridThreeD{6}{0}{black!75}{4}{5}
            \node[label={\small $p$}] at (2,3) [circle, fill=black!50] {};
            \node[label={\small $p'$}] at (4,5) [circle, fill=black!50] {};
            \node[label={\small $p_2$}] at (3,4) [circle, fill=black!50] {};
            \node[label={\small $p_1$}] at (2,4) [circle, fill=black!50] {};

            \node[label={\small $p$}] at (8,3) [circle, fill=black!50] {};
            \node[label={\small $p_1$}] at (10,5) [circle, fill=black!50] {};
            \node[label={\small $p_2$}] at (9,4) [circle, fill=black!50] {};

            \draw[red!50, thin](4.25,0.75) -- ++ (-4.5,4.5);
            \draw[black!50, thin](4.25,1.75) -- ++ (-3.5,3.5);
            \draw[black!50, thin](4.25,2.75) -- ++ (-2.5,2.5);
            \draw[black!50, thin](4.25,3.75) -- ++ (-1.5,1.5);
            
            \draw[red!50, thin](10.25,0.75) -- ++ (-4.5,4.5);
            \draw[black!50, thin](10.25,1.75) -- ++ (-3.5,3.5);
            \draw[black!50, thin](10.25,2.75) -- ++ (-2.5,2.5);
            \draw[black!50, thin](10.25,3.75) -- ++ (-1.5,1.5);
    
        \end{tikzpicture}
        
        \caption{Note that the location of $p'$ on the left or $p_1$ on the right are not uniquely determined. }
        \label{fig6.9}
    \end{figure}
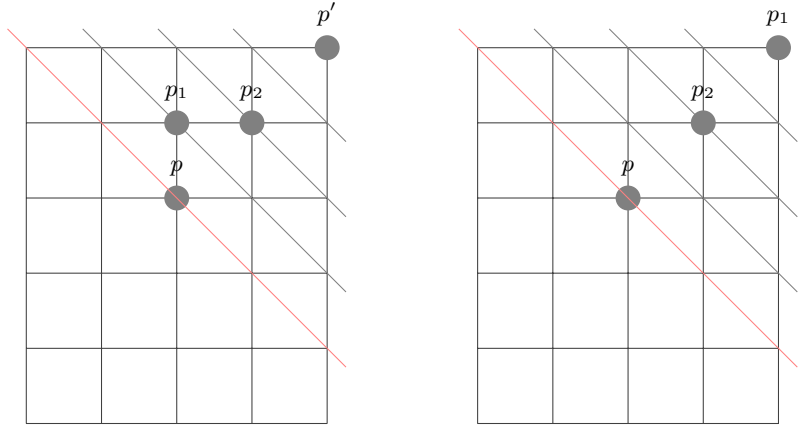

  If $p_3$ cannot ascend, the sequence of moves must be suboptimal, so assume $p_3$ can ascend, the third ascent so far in the moveset. We now consider cases based on which border the fourth ascent occurs from. If the fourth ascent occurs from $l_4$ or any even numbered border, these pieces cannot hop over $p$, $p_2$, or $p'$($p_1$) in the left-hand (resp. right-hand) cases. \\
    
    In the left-hand case from Figure \ref{fig6.9}, before the next ascent, a piece must move to the border between $p_2$ and $p'$, $p_2$ must move for $p_1$ to be free to be hopped over, and $p'$ must move so the frontmost piece can be hopped over. These must necessarily be performed as front or back pushes for optimality, however if these are performed as front or back pushes, this is at least 3 separate moves, making the sequence suboptimal. \\
    
    In the right-hand case from Figure \ref{fig6.9}, pieces must move to the border between $p$ and $p_2$ and the border between $p_2$ and $p_1$, and $p_1$ must move so the frontmost piece can be hopped over. Again, these must be performed as front or back pushes for optimality, but as front or back pushes, this is at least 3 separate moves, making the sequence suboptimal.\\
    
    Finally, suppose the fourth ascent occurs from $l_5$ or beyond. There is necessarily at least one piece on $l_4$ which must front push, and $p_3$ must move to open the ladder it climbed. If there are two pieces on $l_4$, then there are 3 moves that must occur which is suboptimal, so suppose there is only one piece on $l_4$, $p_4$, which both $p_3$ and $p_1$ hopped over. By similar arguments to earlier, $p_4$ cannot hop forward, it must shift to $l_5$. Now, if there are 2 or more pieces on $l_5$, then suboptimality is forced, as one must move before the ascent, which totals 3 moves before the ascent. However, it is possible that $l_5$ contains only one piece, $p_4$, as $p_5$ as shown in Figure \ref{fig6.10} could have been the piece to front push prior to the ascent of $p_3$.  \\
    
    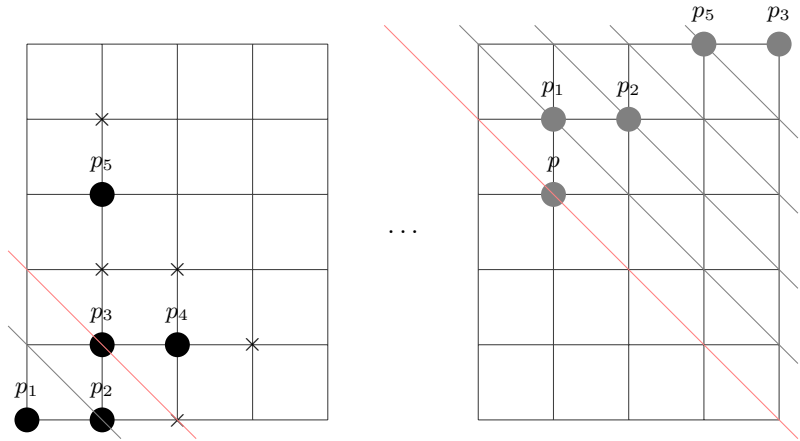
\begin{figure}[h]
        \centering
        
        \begin{tikzpicture}
            \gridThreeD{0}{0}{black!75}{4}{5}
            \gridThreeD{6}{0}{black!75}{4}{5}
            \node[label={\small $p_1$}] at (0,0) [circle, fill=black] {};
            \node[label={\small $p_2$}] at (1,0) [circle, fill=black] {};
            \node[label={\small $p_3$}] at (1,1) [circle, fill=black] {};
            \node[label={\small $p_4$}] at (2,1) [circle, fill=black] {};
            \node[label={\small $p_5$}] at (1,3) [circle, fill=black] {};

            \node[label={\small $p$}] at (7,3) [circle, fill=black!50] {};
            \node[label={\small $p_5$}] at (9,5) [circle, fill=black!50] {};
            \node[label={\small $p_2$}] at (8,4) [circle, fill=black!50] {};
            \node[label={\small $p_1$}] at (7,4) [circle, fill=black!50] {};
            \node[label={\small $p_3$}] at (10,5) [circle, fill=black!50] {};
            
            \node at (5,2.5) {$\ldots$};
            \node at (1,2) {$\times$};
            \node at (2,0) {$\times$};
            \node at (3,1) {$\times$};
            \node at (2,2) {$\times$};
            \node at (1,4) {$\times$};
    
            \draw[black!50, thin](1.25,-.25) -- ++ (-1.5,1.5);
            \draw[red!50, thin](2.25,-.25) -- ++ (-2.5,2.5);
            
            \draw[red!50, thin](10.25,-0.25) -- ++ (-5.5,5.5);
            \draw[black!50, thin](10.25,0.75) -- ++ (-4.5,4.5);
            \draw[black!50, thin](10.25,1.75) -- ++ (-3.5,3.5);
            \draw[black!50, thin](10.25,2.75) -- ++ (-2.5,2.5);
            \draw[black!50, thin](10.25,3.75) -- ++ (-1.5,1.5);
    
        \end{tikzpicture}
        
        \caption{On the left, the necessary setup if $l_4$ and $l_5$ only have one piece, and on the right, in the lone scenario where $l_5$ front pushed before $p_3$ ascended. }
        \label{fig6.10}
    \end{figure}
    
    In this case, then $p_4$ can ascend and optimality is still preserved. However, there must be two different pieces on $l_6$, the piece $p_5$ hops over when it front pushes, and the piece $p_3$ hops over after hopping over $p_4$ when it ascends. By a similar argument as earlier, the next ascent occurring from $l_6$ forces suboptimality, and it is straightforward to see that at least 3 unique moves must occur for the next ascent to occur from $l_7$ or beyond. Thus, if after the initial 2 ascents, the next ascent occurs from $l_3$, and $l_3$ only contained one piece, suboptimality is forced. We have exhausted all cases, and conclude that $\omega(A_i) \leq 0$.

    \end{proof}

    \begin{theorem}
        If $C$ is a configuration with 5 or more pieces, then $C$ has speed less than or equal to $2/3$.
    \end{theorem}
    
    \begin{proof}
        Consider any $m$-move trajectory $M$ of $C$. If $M$ has no consecutive ascents, then Theorem 3.2 implies $M$ has speed at most 2/3. Otherwise, perform an isolating partition of $M = \sum_{i=1}^n A_i$. By the two previous lemmas, $\omega(M) = \omega(\sum A_i) = \sum \omega(A_i) \leq 0$, as desired. 
    \end{proof}
    
    \medskip\medskip
    
   \textbf{Acknowledgment.} The authors are grateful to the referees for many valuable suggestions and corrections.

\end{document}